\documentclass[a4paper,12pt]{article}

\usepackage[utf8]{inputenc}
\usepackage{amstext}
\usepackage{amsfonts}
\usepackage{amsthm}
\usepackage{textcomp}
\usepackage{amssymb}
\usepackage{amscd}
\usepackage{amsmath}
\usepackage{xcolor}
\usepackage{tikz-cd}
\usepackage{caption}
\usepackage{enumitem}
\usepackage{theoremref}
\usepackage{mathtools}
\usepackage[usestackEOL]{stackengine}
\setlist[enumerate]{label=({\arabic*})}

\newtheorem{defn}{Definition}[subsection]

\newtheorem{exmp}[defn]{Example}

\newtheorem{rmk}[defn]{Remark}

\newtheorem{prop}[defn]{Proposition}

\newcommand{\N}{\mathbb{N}}

\newcommand{\Z}{\mathbb{Z}}

\DeclareMathOperator*\lowlim{\underline{lim}}
\DeclareMathOperator*\uplim{\overline{lim}}

\date{July 26, 2022}
\begin{document}
\title{On the spectrum of weighted shifts}

\author{Emma D'Aniello and Martina Maiuriello}

\newcommand{\Addresses}{{
  \bigskip
  \footnotesize
  
  E.~D'Aniello,\\
  \textsc{Dipartimento di Matematica e Fisica,\\ Universit\`a degli Studi della Campania ``Luigi Vanvitelli",\\
  Viale Lincoln n. 5, 81100 Caserta, ITALIA}\\
  \textit{E-mail address: \em emma.daniello@unicampania.it}\\
\vspace{0.5cm}

  M.~Maiuriello,\\
  \textsc{Dipartimento di Matematica e Fisica,\\ Universit\`a degli Studi della Campania ``Luigi Vanvitelli",\\
  Viale Lincoln n. 5, 81100 Caserta, ITALIA}\\
  \textit{E-mail address: \em martina.maiuriello@unicampania.it}

}}

\maketitle

\begin{abstract} 
It is well-known that, in Linear Dynamics, the most studied class of linear operators is certainly that of weighted shifts, on the separable Banach spaces $c_0$ and $\ell^p$, $1 \leq p < \infty$. 
Over the last decades, the intensive study of such operators has produced an incredible number of versatile, deep and beautiful results, applicable in various areas of Mathematics and the 
relationships between various important notions, especially concerning chaos and hyperbolic properties, and the spectrum of weighted shifts are investigated. In this paper, 
we investigate the point spectrum of weighted shifts and, under some regularity hypotheses on the weight sequence, we deduce the spectrum. 
\end{abstract}

\let\thefootnote\relax\footnote{2010 {\em Mathematics Subject Classification:} Primary: 47B37;  Secondary: 47A10, 47A05.\\
{\em Keywords:} Weighted Shifts, Linear Dynamics, Spectrum, Point Spectrum.}
\tableofcontents
\newpage

\section{Introduction}
It is known that, in 1929, Birkhoff obtained an example of a linear operator having a dense orbit \cite{Birkhoff}. Some years later, the same has been shown for other two types of 
fundamental operators in analysis: the differentiation operators \cite{Maclane} and the shifts \cite{Rolewicz}. 
Motivated by these three examples and by the fact that density of orbits is one of the main ingredients of chaos, in the twentieth century, many mathematicians began to focus on 
the dynamical properties of linear operators starting from the analysis of the three ones mentioned above, thus leading to the birth of the fascinating area of Mathematics known 
as {\em{Linear Dynamics}}. Over the years, the usefulness of the Birkhoff operators, of the differentiation operators and of the shifts, has become undoubted, making them a model 
for understanding the behaviors of more complex operators.
This is particularly illustrated by the shifts or, more generally, the weighted shifts, whose flexibility and ``simplicity of action'' have made them not only the most studied class of linear 
operators, but also an excellent tool to clearly illustrate definitions and create crucial counterexamples in Linear Dynamics. In particular, weighted shifts are known to be essential for the 
comprehension of some classes of linear operators, among which {\em{composition operators}} \cite{DAnielloDarjiMaiurielloShift}, just to cite one of the most famous and versatile examples.
Hence, starting in 1995 with the characterization of hypercyclic weighted shifts \cite{S}, their dynamical properties have been so studied that, nowadays, the most of the fundamentals of Linear 
Dynamics (like hypercyclicity, mixing, chaos, Li-Yorke chaos, expansivity, hyperbolicity and shadowing) are completely characterized for such operators 
\cite{BernardesCiriloDarjiMessaoudiPujalsJMAA2018, BDP, D'AnielloDarjiMaiuriello, DarjiPires}. See \cite{GrosseErdmann} and \cite{tesiMM2022} for a global overview on the subject. \\
As it is often the case in Operator Theory, it may be useful to know the structure of the spectrum and to impose some conditions on it, to ensure that a certain dynamical property manifests. 
For instance, it is known that an invertible operator $T$ on $\mathbb{C}^n$ is expansive if and only if the point spectrum of $T$ does not intersect  the unit circle ${\mathbb T}$ \cite[Theorem 1]{Eisenberg}. 
Subsequently, as it is well-known that an invertible operator $T$ on a Banach space $X$ is hyperbolic if its spectrum does not intersect ${\mathbb T}$, it was shown 
in \cite[Theorem D]{BernardesCiriloDarjiMessaoudiPujalsJMAA2018} and \cite[Theorem 1]{M} that $T$ is uniformly expansive if and only if the same condition is satisfied by the 
approximate point spectrum, providing, as a corollary, that invertible hyperbolic operators are uniformly expansive. Moreover, the above requirement on the spectrum is a necessary condition to 
get Li-Yorke chaos \cite[Corollary 6]{BermudezBonillaMartinezPeris2011}. As hyperbolicity, also generalized hyperbolicity is defined, for an invertible operator $T$ on a Banach space $X$, in terms of 
the spectrum of certain restrictions of the operators.  These results are just a glimpse of the close relationship between the spectrum of an operator $T$ and its dynamical properties and, also in the specific 
case of weighted shifts, it turns out that the spectrum plays a key rule in the relative theory. \\
\indent
In this paper we focus on weighted shifts on $\ell^p$-spaces: we generalize some results, proved on the Hilbert space $\ell^2$, 
to weighted shifts defined on the Banach space $\ell^p$, $1\leq p <\infty.$  The present article aims, after recalling basic and fundamental results about the spectrum of shifts and weighted shifts 
(Section 2), to investigate the point spectrum of weighted shifts and, under some regularity hypotheses on the weight sequence, to deduce the spectrum (Section 3). \\

Throughout the article, as usual, ${\mathbb N}$ denotes the set of all positive integers; ${\mathbb Z}$ denotes the set of all integers; ${\mathbb D}$ and ${\mathbb T}$ are, 
respectively, the open unit disk and the unit circle in the complex plane $\mathbb C$. 

\section{Definitions and Background Results}

In this section, we fix some terminology and recall some basic notions and results concerning the spectral theory of, first, a bounded linear operator on a complex Banach space 
$X$ and, then, more specifically, of shifts  \cite{Dowson, Kubrusly, Lay}. 
In the sequel, ${\cal{L}}(X)$ denotes the algebra of all bounded linear operators from the complex Banach space $X$ into itself.

\subsection{Spectrum and Spectral Radius}
\begin{defn}
Let $T \in {\cal{L}}(X).$ The {\em spectrum $\sigma (T)$} of the operator $T$ is the set \[ \sigma (T)= \{ \lambda \in {\mathbb C} : T- \lambda I \text{ is not invertible in } {\cal{L}}(X)\}. \] 
\end{defn}

\begin{defn}
Let $T \in {\cal{L}}(X).$ The {\em spectral radius $r(T)$} of the operator $T$ is defined by \[ r(T)= \sup \{ \vert \lambda \vert : \lambda \in \sigma(T)\}. \] 
\end{defn}

It is well-known that the spectral radius $r(T)$ satisfies the {\em spectral radius formula} $r(T)=\lim_{n \rightarrow \infty} \Vert T^n \Vert ^{\frac{1}{n}}$ and hence $r(T) \leq \Vert T \Vert$, 
and that $\sigma(T)$ is a non-empty compact subset of ${\mathbb C}$ contained in the ball $\{ \lambda \in {\mathbb C} : \vert \lambda \vert \leq \Vert T \Vert \}$. Moreover, if $T$ is invertible then 
\[ \sigma (T^{-1})= \left \{ \frac{1}{\lambda} : \lambda \in \sigma (T) \right \}, \] and so \[ \frac{1}{r(T^{-1})} = \inf \{ \vert \lambda \vert : \lambda \in \sigma(T) \}, \] and hence, in this case the spectrum 
$\sigma(T)$ of $T$ lies in the annulus $\{ \lambda \in {\mathbb C} : \frac{1}{r(T^{-1})} \leq \vert \lambda \vert \leq r(T) \}$.\\

We now recall the definitions of some subsets of the spectrum which will be used in the sequel.

\begin{defn} \label{defspectra}
Let $T \in {\cal{L}}(X).$
\begin{enumerate}
\item The set of all $\lambda \in \sigma(T)$ such that $T-\lambda I$ is not one-to-one is called the {\em point spectrum} of $T$ and is denoted by $\sigma _p(T)$. 
It follows that $\sigma _p(T)$ consists of all the {\em eigenvalues} of $T$. \\
\item The set of all $\lambda \in \sigma(T)$ such that $T-\lambda I$ is not Fredholm, (an operator is Fredholm if its range is closed and both its kernel and its cokernel 
are finite-dimensional) is called the {\em essential spectrum} 
of $T$ and is denoted by $\sigma _{e}(T)$.
\end{enumerate}
\end{defn}

The essential spectrum $\sigma_e(T)$ is a closed subset of $\sigma(T)$  and, clearly, it is empty when the space $X$ is finite-dimensional \cite{LaursenNeumann}. Moreover, it is 
invariant under compact perturbation, as the following result shows:  

\begin{prop}{\em \cite[Theorem 4.1]{EdmundsEvans}} \label{propcompact} 
Let $T \in {\cal L}(X)$. If $S$ is a compact operator then \[ \sigma_e(T+S)=\sigma_e(T).\]
\end{prop}

We recall the following useful result:
\begin{prop}{\em \cite[Proposition 1.F]{Kubrusly}} \label{invertible}
Let $T$ be a bounded operator on a Banach space $X$ . If there exists an invertible operator $S:X \rightarrow X$ for which $\Vert S - T\Vert < \Vert S^{-1}\Vert ^{-1}$, then $T$ is itself invertible.
\end{prop}

We conclude this paragraph by recalling that the above definitions of spectrum can be extended to any bounded linear operator $T$ acting on a real Banach space $X$ via its complexification 
$T_{\mathbb C}$, in the sense that  we define the spectrum of $T$ as the set of all $\lambda \in \mathbb C$ such that $T_{\mathbb C} - \lambda I$ is not invertible as an operator acting on 
$X_{\mathbb C}$, i.e. we define $\sigma(T)=\sigma(T_{\mathbb C}).$

\subsection{Weighted Shifts}

We recall some preliminary definitions and results.

\begin{defn} 
Let $X=\ell^p({\mathbb N})$, $1 \leq p < \infty$ or $X=c_0({\mathbb N}).$ Let  $w=\{w_n\}_{n \in {\mathbb N}}$ be a bounded sequence of scalars, called {\em weight sequence}.  Then,\\
\begin{itemize}
\item the {\em unilateral weighted forward shift $F_w:X \rightarrow X$} is defined by \[F_w( \{x_n\}_{n \in \mathbb N }) =\{w_{n-1}x_{n-1}\}_{n \in \mathbb N}, \]
meaning $$ F_w(\{x_1,x_2,...\})=\{0,w_1x_1,w_2x_2,...\};$$ 
\item the {\em unilateral weighted backward shift $B_w:X \rightarrow X$} is defined by \[B_w( \{x_n\}_{n \in {\mathbb N}}) =\{w_{n+1}x_{n+1}\}_{n \in {\mathbb N}},\] meaning $$ B_w(\{x_1,x_2,...\})=\{w_2x_2,w_3x_3,...\}.$$
\end{itemize}
If, instead of $\mathbb N$, we consider ${\mathbb Z}$, the shift is called {\em bilateral}. 
\end{defn}
Clearly, a weighted shift (unilateral or bilateral) is injective if and only if none of the weights is zero, and, a bilateral weighted shift is invertible if and only if  $\inf_{n \in \mathbb Z} \vert w_n \vert >0$. 
Of course, a unilateral weighted shift is never invertible. 

\begin{rmk}
Let $w=\{w_n\}_{n \in {\mathbb Z}}$ be a weight sequence with  $\inf_{n \in \mathbb Z} \vert w_n \vert >0$. Obviously, $B_w^{-1}=F_{\tilde{w}}$ and $F_w^{-1}=B_{\tilde{w}}$, where $\tilde{w}=\{\frac{1}{w_n}\}_{n \in \mathbb Z}$.
\end{rmk}

\begin{rmk} \label{rmknorm} Let $w=\{w_n\}_{n \in  {\mathbb Z}}$,  be a weight sequence. Let $T$ be the bilateral weighted shift $F_w$ or $B_w$ on $X=\ell^p(\Z)$, $1 \leq p < \infty$, or $X=c_0(\Z)$. It will be useful for the sequel to note that, for every $n \in {\mathbb N}$,   
\begin{enumerate}
\item{$\Vert T^n \Vert = \sup_{k \in \Bbb Z} \vert w_kw_{k+1} \cdots w_{k+n-1}\vert$}
\item{if $T$ is invertible, then $\Vert T^{-n} \Vert = \sup_{k \in \Bbb Z} \vert w_kw_{k+1} \cdots w_{k+n-1}\vert^{-1}$}
\end{enumerate}
\begin{proof}
We only show $(1)$ for $T=F_w$, as the rest follows in a similar fashion.  Given  $F_w$ on $X=\ell^p(\Z)$, $1 \leq p < \infty$, then
\begin{eqnarray*}
\Vert F_w^n(\{x_k\}_{k \in \Z}) \Vert_p &=& \Vert \{ w_{k-n} \cdots w_{k-1}x_{k-n} \}_{k \in \Z}  \Vert_p =\left( \sum_{k \in \Z} \vert w_{k-n} \cdots w_{k-1}x_{k-n} \vert ^p\right)^{\frac{1}{p}}\\
&\leq& \sup_{k \in \Z} \vert w_{k-n} \cdots w_{k-1} \vert \left( \sum_{k \in \Z} \vert x_{k-n} \vert ^p\right)^{\frac{1}{p}} = \sup_{k \in \Z} \vert w_{k} \cdots w_{k+n-1} \vert \left( \sum_{k \in \Z} \vert x_{k} \vert ^p\right)^{\frac{1}{p}}\\
&=& \sup_{k \in \Z} \vert w_{k} \cdots w_{k+n-1} \vert  \Vert \{x_k\}_{k \in \Z}  \Vert_p, \\
\end{eqnarray*} 
and, on $X=c_0(\Z)$,
\begin{eqnarray*}
\Vert F_w^n(\{x_k\}_{k \in \Z}) \Vert_{\infty} &=& \Vert \{ w_{k-n} \cdots w_{k-1}x_{k-n} \}_{k \in \Z}  \Vert_{\infty} = \sup_{k \in \Z} \vert w_{k-n} \cdots w_{k-1}x_{k-n} \vert \\
&\leq& \sup_{k \in \Z} \vert w_{k-n} \cdots w_{k-1} \vert \sup_{k \in \Z} \vert x_{k-n} \vert = \sup_{k \in \Z} \vert w_{k} \cdots w_{k+n-1} \vert \sup_{k \in \Z} \vert x_{k} \vert \\
&=& \sup_{k \in \Z} \vert w_{k} \cdots w_{k+n-1} \vert  \Vert \{x_k\}_{k \in \Z}  \Vert_{\infty}. \\
\end{eqnarray*} 
Hence,
\begin{eqnarray*}
\Vert F_w^n\Vert &=& \inf_{\{x_k\} \in X} \{ c \geq 0 \ : \ \Vert F_w^n(\{x_k\}_{k \in \Z}) \Vert_X \leq c \Vert \{x_k\}_{k \in \Z} \Vert_X\}\\
&\leq&  \sup_{k \in \Z} \vert w_{k} \cdots w_{k+n-1} \vert .
\end{eqnarray*} 
By computing $F_w^n$ at $e_k=\{\ldots 0,0,\underbracket[0.4pt]{1}_{\clap{\scriptsize$k$-th position}} ,0,0, \ldots\}$, the reverse of the above inequality is obtained. Hence 
\[ \Vert F_w^n \Vert = \sup_{k \in \Bbb Z} \vert w_kw_{k+1} \cdots w_{k+n-1}\vert.\] 
Moreover, if $F_w$ is invertible, an analogous computation gives  
\[\Vert F_w^{-n} \Vert = \sup_{k \in \Bbb Z} \vert w_kw_{k+1} \cdots w_{k+n-1}\vert^{-1}. \]
\end{proof}
\end{rmk}

\begin{rmk}
By replacing $\Z$ by $\N$ in $(1)$ of Remark \ref{rmknorm}, an analogous argument gives the norms of $F_{w}$ and $B_{w}$ in the unilateral case. 
\end{rmk}

\begin{prop}{\em \cite[Proposition 1.6.15]{LaursenNeumann}; \cite[Theorem 4]{Shields}; \cite[Remark 35]{BernardesCiriloDarjiMessaoudiPujalsJMAA2018}}
Let $X=\ell^p({\mathbb N})$, $1 \leq p < \infty$, or $X=c_0({\mathbb N}).$ Let $T$ be the unilateral weighted shift $F_w$ or $B_w$, on $X$. Then, the spectrum 
of $T$ is the disk \[\sigma(T)=\{ \lambda \in {\mathbb C} : \vert \lambda \vert \leq r(T)\}.\]
\end{prop}

\begin{prop}{\em \cite[Theorem 2.1]{Bourhim}; \cite[Theorem 5]{Shields}; \cite[Remark 35]{BernardesCiriloDarjiMessaoudiPujalsJMAA2018}} \label{propspectrumforward2}
Let $X=\ell^p({\mathbb Z})$, $1 \leq p < \infty$, or $X=c_0({\mathbb Z}).$ Let $T$ be the bilateral weighted shift $F_w$ or $B_w$, on $X$. The followings hold. 
\begin{itemize}
\item[a)] If $T$ is non-invertible, then its spectrum is the disk \[\sigma(T)=\{ \lambda \in {\mathbb C} : \vert \lambda \vert \leq r(T)\}.\] 
\item[b)] If $T$ is invertible, then its spectrum is the annulus \[\sigma(T)=\left \{ \lambda \in {\mathbb C} : \frac{1}{r(T^{-1})} \leq \vert \lambda \vert \leq r(T)\right \}.\]
\end{itemize}
\end{prop}

The followings are well-known results which will be useful in the sequel.

\begin{prop}{\em \cite[Exercise 5.2.10]{GrosseErdmann}, \cite[Proposition 1.6.14]{LaursenNeumann}  \cite[Proposition 4]{Shields}}\label{compactshift}
Let $X=\ell^p$, $1 \leq p < \infty$, or $X=c_0.$ Let $T$ be the unilateral (resp. bilateral) weighted  $F_w$ or $B_w$. Then, $T$ is compact if and only if $\lim_{n \rightarrow \infty }
 w_n=0$ $(resp. \lim_{\vert n \vert \rightarrow \infty } w_n=0)$.
\end{prop}

When the weights $w=\{w_n\}_{n \in  {A}}$, with $A=\mathbb N$ or $A=\mathbb Z$, are such that $w_n =1$ for each $n \in {A}$, then $F_w$ and $B_w$ reduce to the (unweighted) forward and backward shifts 
denoted by $F$ and $B$, respectively. Clearly, $\Vert F \Vert=\Vert B \Vert=1$, and for this simple case the spectrum and its part were completely analyzed, as partially summarized in the following result (for a detailed description of all the parts of the spectrum, see \cite{Takahiro} and 
\cite[Proposition 2.M]{Kubrusly}, for the case $p=2$; and \cite[Corollary 3.2]{Gohberg}; \cite[Example 3.7.7]{LaursenNeumann}; \cite[Proposition 10.2.8]{Cheng}, for the case 
$1 \leq p < \infty$).

\begin{prop}
Let $B$ and $F$ be the unilateral backward and the unilateral forward shift, respectively. The followings hold.
\begin{enumerate}
\item{$\sigma _p(F)=\emptyset$; $\sigma_e(F)={\mathbb T}$};
\item{$\sigma_p(B)={\mathbb D};$ $\sigma_e(B)={\mathbb T}$; }
\item{$\sigma (F)=\sigma(B)=\{ \lambda \in {\mathbb C} : \vert \lambda \vert \leq 1\}$.}
\end{enumerate}
If $B$ and $F$ denote the bilateral backward and the bilateral forward shift, respectively, then
\begin{enumerate}
\item{$\sigma _p(F)=\emptyset$; $\sigma_e(F)={\mathbb T}$};
\item{$\sigma _p(B)=\emptyset$; $\sigma_{e}(B)={\mathbb T}$;}
\item{$\sigma (F)=\sigma(B) ={\mathbb T}$.}
\end{enumerate}
\end{prop}

The definition of an adjoint operator is given, in general, for operators defined on a Hilbert space. If one tries to transfer this definition to operators on Banach spaces, 
it immediately appears an obstacle: the absence of an inner product.  Hence, given a Banach space $X$, it is necessary to introduce the dual space $X^*$, and to define a new product on 
$X \times X^*$, $\langle \cdot,\cdot \rangle$, as in the following definition.

\begin{defn}
Let $X$ be a Banach space and let $X^*$ be its dual. Let $x^* \in X^*$. Then, for each $x \in X$, we define \[\langle x,x^* \rangle=x^*(x).\]
\end{defn}

\begin{defn}{\em{(Adjoint operator)}} \label{defadjoint}
Let $X$ be a Banach space and let $T \in {\cal L}(X)$. The operator $T^*\in {\cal L}(X^*)$ defined by $T^*x^*=x^* \circ T$, that is \[\langle x,T^*x^* \rangle =\langle Tx,x^* \rangle, \hspace{0.3cm} x \in X, x^* \in X^*,\] 
is the {\em adjoint of $T$.} The operator $T$ is said to be {\em{unitary}} if $T^*=T^{-1}$. 
\end{defn}
The following example will be useful in the sequel.

\begin{exmp}
Let $A= \mathbb N$ or $A=\mathbb Z$. As it is well-known, the dual space of $\ell ^p(A)$, $1\leq p < \infty$, is given by $\ell^q(A)$, where $\frac{1}{p}+\frac{1}{q}=1$ for $1< p < \infty$, 
and $q=\infty$ for $p=1$. The continuous linear functionals $x^*$ on $\ell ^p(A)$ are precisely the maps of the form 
\[x^*(x)= \langle x,x^* \rangle= \sum_{n \in A}x_n\overline{y_n}, \tag{$\spadesuit$}\] with $y=\{y_n\}_{n \in A} \in \ell^q(A)$.  \\
Now, let $A=\mathbb Z$ and let $F_w$ be a bilateral weighted forward shift on $ \ell^p({\mathbb Z}) $, $1\leq p < \infty$. Then, according to 
Definition \ref{defadjoint}, its adjoint is a bilateral weighted backward shift on $ \ell^q({\mathbb Z})$ given by \[F_w^*(e_n)=\overline{w_{n-1}}e_{n-1},\] for 
each $n \in \mathbb Z$ (see, for instance, \cite{MillerMillerNeumann}). In fact, given any $x \in \ell^p({\mathbb Z})$ and $x^* \in \ell^q({\mathbb Z}),$ it is
\begin{eqnarray*}
\langle x,F_w^*x^* \rangle &= & \langle F_wx,x^* \rangle = \sum_{n=-\infty}^{+\infty}(F_wx)_n\overline{y_n}= \sum_{n=-\infty}^{+\infty}w_{n-1}x_{n-1}\overline{y_n} \\
&=& \sum_{n=-\infty}^{+\infty}w_{n}x_{n}\overline{y_{n+1}} = \sum_{n=-\infty}^{+\infty}x_{n}\overline{(\overline{w_{n}}{y_{n+1}})} \\
&=&  \sum_{n=-\infty}^{+\infty}x_{n}B_{\tilde{w}}(\overline{y_{n}}) = \langle x,B_{\tilde{w}}x^* \rangle,\hspace{0.3cm} (\text{where $\tilde{w}=\{\tilde{w}_n\}_n$ with $\tilde{w}_n=\overline{w_{n-1}}$}) 
\end{eqnarray*}
and hence $F_w^*(e_n)=B_{\tilde{w}}(e_n)=\tilde{w}_ne_{n-1}=\overline{w_{n-1}}e_{n-1}$, for each $n \in \mathbb Z$.\\
Analogously, if $A=\mathbb N$ and $F_w$ is a unilateral weighted forward shift on $ \ell^p({\mathbb N}) $, $1\leq p < \infty$, then its adjoint is the unilateral weighted backward shift 
on $ \ell^q({\mathbb N})$ given by $F_w^*(e_0)=0$ and $F_w^*(e_n)=\overline{w_{n-1}}e_{n-1}$, for each $n \geq 1$. Similarly,  $B_w^*(e_n)=\overline{w_{n+1}}e_{n+1}$, 
for each $n \geq 1$.

\end{exmp}

\subsection{Similarity of Operators}
The notion of similarity plays an important role in the theory of bounded linear operators (see \cite{Kubrusly, LaursenNeumann}).
  
\begin{defn}[Similarity] \label{defsimilarity}
Let $X$ be a Banach space. Two operators $T,S \in {\cal L}(X)$ are called {\em similar} if there exists an invertible operator $W \in {\cal L}(X)$ such that $S=W^{-1}TW$.
\end{defn}
In Linear Dynamics, the word ``similarity" is replaced by ``conjugation'', a word originally borrowed from Topological Dynamics.
A special similarity on inner product spaces, i.e., on Hilbert spaces, is given by the unitarily equivalence: 

\begin{defn}[Unitarily Equivalence]
Let $\cal H$ be a Hilbert space. Two operators $T,S \in {\cal L}({\cal H})$ are said to be {\em unitarily equivalent} if there exists a unitary operator $U \in {\cal L}({\cal H})$, such that $S=U^{-1}TU$.
\end{defn}

Note that similarity is a less severe definition than the one of unitarily equivalence.
The importance of similarity and unitarily equivalence follows from the fact that they preserve many properties of an operator. It is known that similarity preserves invariant subspaces, 
that is if two operators are similar, and if one has a nontrivial invariant subspace, then so does the other \cite[Proposition 1.J]{Kubrusly}. Moreover, among the invariants of similarity 
the most important are the spectrum and the spectral radius, as the following result shows:

\begin{prop}{\em \cite[Proposition 2.B]{Kubrusly}}
Similarity preserves the spectrum and its parts, and so it preserves the spectral radius also. That is, in particular, if $X$ is a Banach space and $T,S \in {\cal L}(X)$ are two similar operators, then:
\begin{enumerate}
\item $\sigma_p(T)=\sigma_p(S)$;
\item $\sigma(T)=\sigma(S)$;
\item $r(T)=r(S)$.
\end{enumerate}
\end{prop}

Of course, if two operators on an Hilbert space are unitarily equivalent, then they are also similar. The converse, in general, only holds for normal operators \cite[Proposition 3.I]{Kubrusly}. 
Hence, unitarily equivalent operators on Hilbert spaces have, in particular, the same spectrum and the same point spectrum \cite[Theorem A.11]{Pothoven}. 
Moreover, it is well-known that unitary equivalence preserves the operator norm \cite[Proposition 2.B]{Kubrusly}.  The following result about unitarily equivalent weighted shifts on  $\ell^2$ is well-known. 

\begin{prop}{\em \cite[Proposition 6.2]{Conway}; \cite[Corollary 1]{Shields}} \label{propequivalence}
Let $A= \mathbb N$ or $A=\mathbb Z$ and $X=\ell^2(A)$. Let $T$ be a weighted shift $F_w$ (resp. $B_w$), with weights $w=\{w_n\}_{n \in A}$. Then, $T$ is unitarily equivalent to the weighted shift 
$F_{\tilde{w}}$ (resp. $B_{\tilde{w}}$), where $\tilde{w}=\{\vert w_n \vert\}_{n \in A}.$
\end{prop}

As the following proposition shows, the same result holds for weighted shifts on the Banach space $\ell ^p$, $1\leq p < \infty$, if we consider similarity instead 
of unitarily equivalence. 

\begin{prop}\label{propsimilarity}
Let $A= \mathbb N$ or $A=\mathbb Z$ and $X=\ell^p(A), 1\leq p < \infty.$ Let $T$ be a weighted shift $F_w$ (resp. $B_w$), with weights $w=\{w_n\}_{n \in A}$. 
Then, $T$ is similar to the weighted shift 
$F_{\tilde{w}}$ (resp. $B_{\tilde{w}}$), where $\tilde{w}=\{\vert w_n \vert\}_{n \in A}.$
\end{prop}

\begin{proof}
The proof is showed only for $T=F_w$, as small changes give it for $T=B_w$. Hence, let $T=F_w$.
According to Definition \ref{defsimilarity}, we need to find an invertible operator $W \in {\cal L}(X)$ with $F_{w}=W^{-1}F_{\tilde{w}}W,$ i.e. such that \[WF_{w}=F_{\tilde{w}}W.  \tag{$\bullet$}\]
In order to do that, let us consider the operator $T:X \rightarrow X$ given by $T(e_n)=\gamma_n e_n,$ where $\gamma=\{\gamma_n\}_{n \in A}$ is a bounded sequence of scalars. \\
Note that, for each $\{x_n\}_{n \in A} \in X$, 
\[TF_{w}(\{x_n\}_{n \in A})=T(\{w_{n-1}x_{n-1}\}_{n \in A})=\{\gamma_n w_{n-1}x_{n-1}\}_{n \in A},\] and 
\[F_{\tilde{w}}T(\{x_n\}_{n \in A})=F_{\tilde{w}}(\{\gamma_{n}x_{n}\}_{n \in A})=\{\tilde{w}_{n-1}\gamma_{n-1}x_{n-1}\}_{n \in A}.\] 
Therefore, $TF_{w}=F_{\tilde{w}}T$ if and only if the sequence $\gamma=\{\gamma_n\}_{n \in A}$ satisfies 
\[\gamma_n w_{n-1}x_{n-1}=\tilde{w}_{n-1}\gamma_{n-1}x_{n-1}, \forall n \in A.\tag{$\clubsuit$}\]
Therefore, choosing $\gamma_0=1$ and taking, for $n>0$ 
\begin{equation*}
\gamma_n = \begin{cases}
1 &\text{if $w_{n-1}=0$}\\
\frac{\tilde{w}_{n-1}}{w_{n-1}}\gamma_{n-1} &\text{if $w_{n-1}\neq 0$}
\end{cases}
\end{equation*}
and, for $n\leq0$
\begin{equation*}
\gamma_{n-1} = \begin{cases}
1 &\text{if $\tilde{w}_{n-1}=0$}\\
\frac{w_{n-1}}{\tilde{w}_{n-1}}\gamma_n &\text{if $\tilde{w}_{n-1}\neq 0$},
\end{cases}
\end{equation*}
we have that such a sequence $\{\gamma_n\}_{n \in A}$ satisfies $(\clubsuit)$, that is $TF_{w}=F_{\tilde{w}}T$. \\
Moreover, note that:
\begin{itemize}
\item[$\bullet$]{$T$ is linear;}
\item[$\bullet$]{$T$ is bounded since, for each $n$, $\vert \gamma_n \vert=1$ (as, by hypothesis, $\vert {w}_{n-1} \vert= \vert \tilde{w}_{n-1} \vert$), and then $\Vert T \Vert=\sup_{n \in A}\vert \gamma_n \vert=1$;}
\item[$\bullet$]{$T$ is invertible, as $\inf_{n \in A}\vert \gamma_n\vert =1>0$.  In particular, $T^{-1}(e_n)=\frac{1}{\gamma_n}e_n$.}
\end{itemize}
Taking $W=T$ we obtain an invertible operator $W \in {\cal L}(X)$ such that $WF_{w}=F_{\tilde{w}}W,$ and hence, by Definition \ref{defsimilarity}, $F_{w}$ and $F_{\tilde{w}}$ are similar.
\end{proof}

\section{On the Spectrum of Weighted Shifts}

Up to now, we have not made assumptions about the sign of the weights.
It is well-known, (see \cite[p. 54]{Conway}; \cite[p. 56]{Shields}), that if a weighted shift has a finite number of zero weights, then it is the direct sum of a finite number of finite 
dimensional operators and a weighted shift with nonzero weights. Hence, weighted shifts with a finite number of zero weights lead back to weighted shifts with all weights non-zero. 
For this reason, and by Proposition \ref{propequivalence} and Proposition \ref{propsimilarity}, from now on we only consider positive weights $\{w_n\}_{n \in A}$. \\

In the previous section, Proposition \ref{propspectrumforward2} describes the spectrum of a weighted shifts with $w=\{w_n\}_{n \in A}$, $A={\mathbb N}$ or $A={\mathbb Z}$, a bounded sequence of scalars. 
In this section, we will see more detailed results on the spectrum of $F_w$ and $B_w$, both in the unilateral case and in the bilateral case.  

\subsection{The Unilateral Case}

\begin{prop} \label{propunilat}
Let $T$ be a unilateral weighted shift $F_w$ or $B_w$ on $\ell^p({\mathbb N})$, $1\leq p < \infty$, with $\{w_n\}_{n \in \mathbb N}$ a bounded sequence of positive reals. Let \[ \lim_{n \rightarrow \infty} w_n=w. \] 
Then, the spectral radius of $T$ is \[r(T)=w.\]
\end{prop}

\begin{proof}
The proof is showed for $T=F_w$, as small changes provide the case $T=B_w$. Hence, consider $T=F_w$.
Let $\epsilon >0$ and $\overline{n} \in \mathbb N$ such that $\vert w_n - w \vert < \epsilon$ for each $n\geq \overline{n}$. Then, in particular, for each $n> \overline{n}$, 
\[0< w_{\overline{n}} \cdots w_{n-1} < (w + \epsilon)^{n-\overline{n}}. \]
Therefore, for each $k \in \mathbb N$ and for each $n> \overline{n}$, 
\[0 < w_{k+\overline{n}} \cdots w_{k+n-1} < (w + \epsilon)^{n-\overline{n}}, \]
and, hence, 
\[ 0 < w_{k} \cdots w_{k+\overline{n}-1}w_{k+\overline{n}} \cdots w_{k+n-1} <  w_{k} \cdots w_{k+\overline{n}-1}(w + \epsilon)^{n-\overline{n}}. \]
Then, it follows that
\begin{eqnarray*}
0 \leq r(F_w)= \lim_{n \rightarrow \infty} \Vert F^n_w \Vert^{\frac{1}{n}}&=&\lim_{n \rightarrow \infty}(\sup_{k \in \mathbb N} \vert w_k \cdots w_{k+n-1}\vert)^{\frac{1}{n}}\\
&\leq& \lim_{n \rightarrow \infty}[\sup_{k \in \mathbb N} ( w_k \cdots w_{k+\overline{n}-1})]^{\frac{1}{n}}(w +\epsilon)^{1-\frac{\overline{n}}{n}} \\
&\leq& \lim_{n \rightarrow \infty}(\sup_{k \in \mathbb N}  w_k)^{\frac{\overline{n}}{n}}(w +\epsilon)^{1-\frac{\overline{n}}{n}} \\
&\leq& \lim_{n \rightarrow \infty}M^{\frac{\overline{n}}{n}}(w +\epsilon)^{1-\frac{\overline{n}}{n}}\hspace{0.3cm} (M=\max\{\sup_{k \in \mathbb N} w_k, w+\epsilon\})\\
&=&w +\epsilon,
\end{eqnarray*}
If $w=0$, then, as $\epsilon$ is arbitrary, it follows that $r(F_w)=0=w.$\\
Now, assume that $w>0$. From the previous computation we obtain that $r(F_w)\leq w$. We want to show that $r(F_w)\geq w.$
Let $F$ denote the unilateral  forward shift on $\ell^p({\mathbb N})$, $1\leq p < \infty$, i.e. \[F(\{x_n\}_{n \in \mathbb N})=\{x_{n-1}\}_{n \in \mathbb N}.\]
Consider the operator $F_{\tilde{w}}=F_w-wF$ defined by: 
\[ F_{\tilde{w}}(\{x_n\}_{n \in \mathbb N})=(F_w-wF)(\{x_n\}_{n \in \mathbb N})=\{(w_{n-1}-w)x_{n-1}\}_{n \in \mathbb N}.\] 
Hence, $F_{\tilde{w}}$ is the unilateral weighted forward shift on  $\ell^p({\mathbb N})$, $1\leq p < \infty$, with weights $\tilde{w}=\{w_n-w\}_{n \in \mathbb N}$.
In our case, the hypothesis 
\[\lim_{n \rightarrow \infty} w_n=w\] 
implies that 
\[\lim_{n \rightarrow \infty} \tilde{w_n}=0,\]
which means, by Proposition \ref{compactshift}, that $F_{\tilde{w}}$ is compact.
Therefore, from Proposition \ref{propcompact} it follows that:
\begin{eqnarray*}
\sigma_e(F_w)=\sigma_e(wF+(F_w-wF))&=&\sigma_e(wF+F_{\tilde{w}})\\
&=&\sigma_e(wF) \hspace{0.5cm} \text{(as $F_{\tilde{w}}$ is compact)}\\
&=& w{\mathbb T} \hspace{0.5cm} \\
&=&\{ \lambda : \vert \lambda\vert=w\}.
\end{eqnarray*}
As $\sigma_e(F_w)$ is a closed subset of $\sigma(F_w)$, then it follows \[r(F_w)=\sup\{\vert \lambda \vert, \lambda \in \sigma(F_w)\}\geq w\]
Hence, the spectral radius of $F_w$ is $r(F_w)=w$.
\end{proof}

\begin{rmk} \label{rmkexmp}
Let $X=\ell^p({\mathbb N}) $, $1\leq p < \infty$, and let $T$ denote a unilateral weighted shift $F_w$ or $B_w$, on $X$, 
with $\{w_n\}_{n \in \mathbb N}$ a positive weight sequence. Let 
\[ \lim_{n \rightarrow \infty} w_n=w. \] 
Then, by $(1)$ of Remark \ref{rmknorm} and Proposition \ref{propunilat}, together with the spectral radius formula, we get
 \begin{eqnarray*}
 r(T)= w = \lim_{n \rightarrow \infty} w_n = \lim_{n \rightarrow \infty} {\Vert T^n \Vert}^{\frac{1}{n}} &=&  \lim_{n \rightarrow \infty}  
 {\left(\sup_{k \in \Bbb N} \vert w_kw_{k+1} \cdots w_{k+n-1}\vert\right)}^{\frac{1}{n}}\\
 &=&\lim_{n \rightarrow \infty}  
 {\left[\sup_{k \in \Bbb N}( w_kw_{k+1} \cdots w_{k+n-1}) \right]}^{\frac{1}{n}}.
 \end{eqnarray*}
Take $w_{n} = 3^{{(-1)}^{n}}$. Then
\[\lim_{n \rightarrow \infty}  
 {\left[\sup_{k \in \Bbb N}( w_kw_{k+1} \cdots w_{k+n-1}) \right]}^{\frac{1}{n}}=1\]
but, clearly,  $\{w_{n}\}_{n \in \mathbb N}$ is not regular, with $\varlimsup_{n\rightarrow \infty} w_{n} = 3$ and  $\varliminf_{n\rightarrow \infty} w_{n} = \frac{1}{3}$. In particular, Proposition \ref{propunilat} 
does not hold anymore if we replace $\lim$ with $\varlimsup$.
\end{rmk}

We point out that, in the Hilbert case $p=2$,  Proposition \ref{propunilat} is a consequence of \cite[Proposition 15]{Shields} and \cite[Proposition 6.8 (a)]{Conway}. 
As in \cite[Remark 1.2]{FaourKhalil} and  \cite[Theorem 8]{Shields} for the Hilbert case $p=2$, the point spectrum of a unilateral forward weighted shift on 
$\ell ^p({\mathbb N})$, $1\leq p < \infty$, turns out to be empty, and the point spectrum of a unilateral backward weighted shift on $\ell ^p({\mathbb N})$, 
$1\leq p < \infty$, satisfies  
\[\{0\} \cup \{\lambda \in {\mathbb C}: \vert \lambda \vert <  \tilde{r}\}  \subseteq \sigma_{p}(B_{w}) \subseteq \{\lambda \in {\mathbb C}: \vert \lambda \vert \leq  \tilde{r}\}\]
where $\tilde{r} = \lowlim_{n \rightarrow \infty} {\left(w_{2} \cdots w_{n+1}\right)}^{\frac{1}{n}}.$
The proofs for the general case $1\leq p < \infty$ are analogous to the ones for $p=2$.

\begin{prop}
Let $F_w: \ell^p({\mathbb N}) \rightarrow \ell^p({\mathbb N}) $, $1\leq p < \infty$, be a unilateral weighted forward shift with $\{w_n\}_{n \in \mathbb N}$ a bounded positive 
weight sequence. Then $\sigma _p(F_w)=\emptyset$.
\end{prop}
\begin{proof}
By contradiction, let  $\lambda \in \sigma _p(F_w)$. Then, there exists $\{ x_n\}_{n \in \mathbb N} \in \ell^p(\mathbb N)\setminus \{{0}\}$ eigenvector of $F_w$ corresponding to the 
eigenvalue $\lambda$, i.e., $F_w(\{x_n\}_{n \in \mathbb N})=\{\lambda x_n\}_{n \in \mathbb N}$. By definition, $F_w(\{x_n\}_{n \in \mathbb N})=\{w_{n-1}x_{n-1}\}_{n \in \mathbb N}$ and, hence, 
the coordinates of $\{ x_n\}_{n \in \mathbb N}$ are such that \[\lambda x_1=0 \,\, \text{ and } \,\, w_{n-1} x_{n-1}=\lambda x_{n} \text{ for each } n \geq 2. \tag{$\star$}\] 
By hypothesis, $w_n>0$ for each $n \in {\mathbb N}$, therefore $F_w$ is injective and then $\lambda \neq 0$. Hence, it follows from $(\star)$ that $x_n=0$ for each $n \in {\mathbb N}$. 
This is a contradiction and so it must be $\sigma _p(F_w)=\emptyset$.
\end{proof}

\begin{prop}
Let $B_w: \ell^p({\mathbb N}) \rightarrow \ell^p({\mathbb N}) $, $1\leq p < \infty$, be a unilateral weighted forward shift with $\{w_n\}_{n \in \mathbb N}$ a bounded positive 
weight sequence. Then
\[\{0\} \cup \{\lambda \in {\mathbb C}: \vert \lambda \vert 
<  \tilde{r}\}  \subseteq \sigma_{p}(B_{w}) \subseteq \{\lambda \in {\mathbb C}: \vert \lambda \vert 
\leq  \tilde{r}\}\]
where $\tilde{r} = \lowlim_{n \rightarrow \infty} {\left(w_{2} \cdots w_{n+1}\right)}^{\frac{1}{n}}.$
\end{prop}
\begin{proof}
Let $\lambda \in \sigma_p(B_w)$. Then, there exists $\{ x_n\}_{n \in \mathbb N} \in \ell^p(\mathbb N)\setminus \{{0}\}$ eigenvector of $B_w$ corresponding to the 
eigenvalue $\lambda$, i.e., $B_w(\{x_n\}_{n \in \mathbb N})=\{\lambda x_n\}_{n \in \mathbb N}$. By definition, $B_w(\{x_n\}_{n \in \mathbb N})=\{w_{n+1}x_{n+1}\}_{n \in \mathbb N}$ and, hence, 
the coordinates of $\{ x_n\}_{n \in \mathbb N}$ are such that \[w_{n+1} x_{n+1}=\lambda x_{n} \text{ for each } n \geq 1. \tag{$\star$}\] 
Therefore, assuming without loss of generality $x_{1}=1$, for each $n \in {\mathbb N}$, we have
\[x_{n+1} = \frac{{\lambda}^{n}}{w_{2} \cdots w_{n+1}}.\]
Then
 \[{\Vert x \Vert}_{p}^{p}   =   \sum_{n=1}^{\infty} {\vert x_{n} \vert}^p  = 1+ \sum_{n = 2}^{\infty} {\left( \frac{{\vert \lambda \vert}^{n-1}}{w_{2} \cdots w_{n}}  \right)}^p.\]
By applying the Cauchy-Hadamard criterion,  as 
\[\uplim_{n \rightarrow \infty} \sqrt[n]{\frac{{\vert \lambda \vert}^{n-1}}{w_{2} \cdots w_{n}}  }  = \frac{\vert \lambda \vert}{\tilde{r}},\]
we obtain the thesis.
\end{proof}

\subsection{The Bilateral Case}

In this section we focus on bilateral weighted shifts. We first generalize a result on the point spectrum proved, for the Hilbert case $p=2$, in \cite[Proposition 6.8 (b)]{Conway} and \cite[Theorem 9]{Shields}, to $1 \leq p < \infty$. Then, under some regularity hypotheses on the weight sequence, we deduce the spectrum.

\begin{prop} \label{propsigmap}
Let $\{w_n\}_{n \in \mathbb Z}$ be a bounded positive weight sequence. Assume that \[ \lim_{n \rightarrow \infty} w_n=w_+ \text{ and }  \lim_{n \rightarrow \infty} w_{-n}=w_-.\] Let $F_w$ and $B_w$ be the bilateral weighted forward shift and the bilateral weighted backward shift, respectively, on $\ell^p({\mathbb Z}) $, $1\leq p < \infty$. Then the following hold.
\begin{enumerate}
\item If $w_+ \leq w_-$, then $\{ \lambda : w_+ < \vert \lambda \vert < w_-\} \subseteq \sigma_p(F_w) \subseteq \{ \lambda : w_+ \leq \vert \lambda \vert \leq w_-\}$.
\item If $w_- < w_+$, then $\sigma_p(F_w)=\emptyset$.
\item If $w_- \leq w_+$, then $\{ \lambda : w_- < \vert \lambda \vert < w_+\} \subseteq \sigma_p(B_w) \subseteq \{ \lambda : w_- \leq \vert \lambda \vert \leq w_+\}$.
\item If $w_+ < w_-$, then $\sigma_p(B_w)=\emptyset$.
\end{enumerate}
\end{prop}

\begin{proof}
(1).  We separate the two cases, Case 1.A: $w_+ >0$ and Case 1.B: $w_+=0$. \\

{\bf{Case 1.A}} Assume $w_+ >0$. We start by showing that $\sigma_p(F_w) \subseteq \{ \lambda : w_+ \leq \vert \lambda \vert \leq w_-\}$. Let $\lambda \in \sigma_p (F_w)$. Let $x=\{ x_n\}_{n \in \mathbb Z}$ be an eigenvector corresponding to $\lambda$, i.e. $F_w(\{x_n\}_{n \in \mathbb Z})=\{\lambda x_n\}_{n \in \mathbb Z}$. By definition, $F_w(\{x_n\}_{n \in \mathbb Z})=\{w_{n-1}x_{n-1}\}_{n \in \mathbb Z}$ and hence the coordinates of $x=\{ x_n\}_{n \in \mathbb Z}$ are such that \[w_{n-1} x_{n-1}=\lambda x_{n} \text{ for each } n \in {\mathbb Z}.\] Therefore, for each $n \in \mathbb N$, we have \[ x_n=\frac{w_0 \cdots w_{n-1}}{\lambda ^n} x_0\,\, \text{ and } \,\, x_{-n}=\frac{\lambda ^n}{w_{-n}\cdots w_{-1}} x_0.\]
Then
\begin{eqnarray*}
\Vert x \Vert_p^p =  \sum_{n=-\infty}^{+ \infty} \vert x_n \vert ^p & = & \sum_{n=1}^{+ \infty} \vert x_{-n} \vert ^p + \vert x_0 \vert ^p + \sum_{n=1}^{+ \infty} \vert x_n \vert ^p \\
&=& \sum_{n=1}^{+ \infty} \vert \lambda \vert^{pn}\vert x_0 \vert^p (w_{-n}\cdots w_{-1})^{-p} + \vert x_0\vert^p + \sum_{n=1}^{+ \infty} \vert \lambda \vert^{-pn} \vert x_0 \vert^p (w_{0}\cdots w_{n-1})^{p}\\
&=&\vert x_0\vert^p \left[\sum_{n=1}^{+ \infty} \vert \lambda \vert^{pn} (w_{-n}\cdots w_{-1})^{-p} + 1 + \sum_{n=1}^{+ \infty} \vert \lambda \vert^{-pn} (w_{0}\cdots w_{n-1})^{p} \right] \hspace{0.2cm} {(\spadesuit)}
 \end{eqnarray*}
 
 By contradiction, assume $\vert \lambda \vert < w_+$. Choose $\epsilon $ such that $\vert \lambda \vert < w_+ - \epsilon$. 
 As  $\lim_{n \rightarrow \infty} w_n=w_+$, let  $\overline{n}$ be so large that $w_n > w_+ - \epsilon$ for each $n\geq \overline{n}$. 
 Then, for each $n> \overline{n}$, \[w_{\overline{n}} \cdots w_{n-1} > (w_+ - \epsilon)^{n-1-\overline{n}+1}=(w_+ - \epsilon)^{n-\overline{n}}. \]
 Therefore
 \begin{eqnarray*}
 \infty > \Vert x \Vert _p^p & \geq &  \sum_{n=1}^{+ \infty} \vert \lambda \vert^{-pn} (w_{0}\cdots w_{n-1})^{p} \\
  & = &  \sum_{n=1}^{\overline{n}} \vert \lambda \vert^{-pn} (w_{0}\cdots w_{n-1})^{p} +  \sum_{n=\overline{n}+1}^{+ \infty} \vert \lambda \vert^{-pn} (w_{0}\cdots w_{\overline{n}-1} \cdots w_{n-1})^{p} \\
  &>& \sum_{n=1}^{\overline{n}} \vert \lambda \vert^{-pn} (w_{0}\cdots w_{n-1})^{p} + (w_{0}\cdots w_{\overline{n}-1})^p  \sum_{n=\overline{n}+1}^{+ \infty} \vert \lambda \vert^{-pn} (w_+ - \epsilon)^{p(n-\overline{n})}\\
  &=&  \sum_{n=1}^{\overline{n}} \vert \lambda \vert^{-pn} (w_{0}\cdots w_{n-1})^{p} + (w_{0}\cdots w_{\overline{n}-1})^p (w_+ - \epsilon)^{-p\overline{n}}  \sum_{n=\overline{n}+1}^{+ \infty} \vert \lambda \vert^{-pn} (w_+ - \epsilon)^{pn}.\\
 \end{eqnarray*}
 
As $\vert \lambda \vert < w_+ - \epsilon$, then the geometric series on the right diverges: this is a contradiction. Hence, it must be $\vert \lambda \vert \geq w_+$.  \\
Analogously, by contradiction, assume $\vert \lambda \vert > w_-$. Choose $\epsilon $ such that $\vert \lambda \vert > w_- + \epsilon$.  As $\lim_{n \rightarrow \infty} w_{-n}=w_-$, let $\tilde{n}$ 
be so large that $w_{-n} < w_- + \epsilon$ for each $n\geq \tilde{n}$. 
Then, for each $n> \tilde{n}$, 
\[w_{-\tilde{n}-1} \cdots w_{-n} < (w_- + \epsilon)^{n-\tilde{n}}. \]
Therefore 
 \begin{eqnarray*}
 \infty > \Vert x \Vert _p^p & \geq & \sum_{n=1}^{+ \infty} \vert \lambda \vert^{pn} (w_{-n}\cdots w_{-1})^{-p} \\
  & = &  \sum_{n=1}^{\tilde{n}} \vert \lambda \vert^{pn} (w_{-n}\cdots w_{-1})^{-p} +  \sum_{n=\tilde{n}+1}^{+ \infty} \vert \lambda \vert^{pn} (w_{-n}\cdots w_{-\tilde{n}} \cdots w_{-1})^{-p} \\
  &>& \sum_{n=1}^{\tilde{n}} \vert \lambda \vert^{pn} (w_{-n}\cdots w_{-1})^{-p} + (w_{-\tilde{n}}\cdots w_{-1})^{-p}  \sum_{n=\tilde{n}+1}^{+ \infty} \vert \lambda \vert^{pn} (w_- + \epsilon)^{-p(n-\tilde{n})}\\
  &=& \sum_{n=1}^{\tilde{n}} \vert \lambda \vert^{pn} (w_{-n}\cdots w_{-1})^{-p} + (w_{-\tilde{n}}\cdots w_{-1})^{-p} (w_- + \epsilon)^{p\tilde{n}}  \sum_{n=\tilde{n}+1}^{+ \infty} \vert \lambda \vert^{pn} (w_- + \epsilon)^{-pn}.\\
 \end{eqnarray*}
 
 As $\vert \lambda \vert > w_- + \epsilon$, then the geometric series on the right diverges.  This is a contradiction. So, it must be $\vert \lambda \vert \leq w_-$.
Hence, we have just proved that, in the case $w_+>0$, it is \[\sigma_p(F_w) \subseteq \{ \lambda : w_+ \leq \vert \lambda \vert \leq w_-\}.\] 
Now, we prove the first inclusion in (1). That is, we show that if $\lambda$ is such that $w_+ < \vert \lambda \vert < w_-$, then  $\lambda$ is an eigenvalue of the vector ${x}=\{x_n\}_{n \in \mathbb Z}$ 
defined by choosing $x_0 \neq 0$ and $w_{n-1} x_{n-1}=\lambda x_{n}$  for each  $n \in {\mathbb Z}$. In fact, choose $\epsilon$ such that $w_+ + \epsilon < \vert \lambda \vert <  w_- - \epsilon$. 

Let  $\overline{n}$ be so large that $w_- - \epsilon < w_{- n}$ and $w_{n} < w_+ + \epsilon$ for each $n\geq \overline{n}$. 
 Then, for each $n> \overline{n}$, \[w_{\overline{n}} \cdots w_{n-1} < (w_+ + \epsilon)^{n-1-\overline{n}+1}=(w_+ + \epsilon)^{n-\overline{n}} \]
 and 
 \[w_{-\tilde{n}-1} \cdots w_{-n} > (w_- - \epsilon)^{n-\tilde{n}}. \]
 Therefore
 \begin{eqnarray*}
  \sum_{n=1}^{+ \infty} \vert \lambda \vert^{-pn} (w_{0}\cdots w_{n-1})^{p} & = &  \sum_{n=1}^{\overline{n}} \vert \lambda \vert^{-pn} (w_{0}\cdots w_{n-1})^{p} +  
  \sum_{n=\overline{n}+1}^{+ \infty} \vert \lambda \vert^{-pn} (w_{0}\cdots w_{\overline{n}-1} \cdots w_{n-1})^{p} \\
  &<& \sum_{n=1}^{\overline{n}} \vert \lambda \vert^{-pn} (w_{0}\cdots w_{n-1})^{p} + (w_{0}\cdots w_{\overline{n}-1})^p  \sum_{n=\overline{n}+1}^{+ \infty} \vert \lambda \vert^{-pn} (w_+ + \epsilon)^{p(n-\overline{n})}\\
  &=&  \sum_{n=1}^{\overline{n}} \vert \lambda \vert^{-pn} (w_{0}\cdots w_{n-1})^{p} + \\
  & &  + (w_{0}\cdots w_{\overline{n}-1})^p (w_+ + \epsilon)^{-p\overline{n}}  \sum_{n=\overline{n}+1}^{+ \infty} \vert \lambda \vert^{-pn} (w_+ + \epsilon)^{pn}\\
  & <& \infty \hspace{0.5cm} (\text{as } \vert \lambda \vert > w_+ + \epsilon)
 \end{eqnarray*}
 
 and 
 \begin{eqnarray*}
  \sum_{n=1}^{+ \infty} \vert \lambda \vert^{pn} (w_{-n}\cdots w_{-1})^{-p} & = &  \sum_{n=1}^{\tilde{n}} \vert \lambda \vert^{pn} (w_{-n}\cdots w_{-1})^{-p} +  \sum_{n=\tilde{n}+1}^{+ \infty} \vert \lambda \vert^{pn} (w_{-n}\cdots w_{-\tilde{n}} \cdots w_{-1})^{-p} \\
  &<& \sum_{n=1}^{\tilde{n}} \vert \lambda \vert^{pn} (w_{-n}\cdots w_{-1})^{-p} + \\
  & & + (w_{-\tilde{n}}\cdots w_{-1})^{-p}  \sum_{n=\tilde{n}+1}^{+ \infty} \vert \lambda \vert^{pn} (w_- - \epsilon)^{-p(n-\tilde{n})}\\
  &=& \sum_{n=1}^{\tilde{n}} \vert \lambda \vert^{pn} (w_{-n}\cdots w_{-1})^{-p} + \\
  & &  + (w_{-\tilde{n}}\cdots w_{-1})^{-p} (w_- - \epsilon)^{p\tilde{n}}  \sum_{n=\tilde{n}+1}^{+ \infty} \vert \lambda \vert^{pn} (w_- - \epsilon)^{-pn}\\
   & <& \infty \hspace{0.5cm} (\text{as } \vert \lambda \vert < w_- - \epsilon).
 \end{eqnarray*}
It follows from $(\spadesuit)$ that $\Vert x \Vert_p < \infty$. We conclude that the vector ${x}=\{x_n\}_{n \in \mathbb Z}$ is such that
\begin{itemize}
\item{$w_{n-1} x_{n-1}=\lambda x_{n}$  for each  $n \in {\mathbb Z}$;}
\item{${x} \in \ell^p({\mathbb Z})\setminus \{0\};$ }
\end{itemize}
i.e. it is an eigenvector of $\lambda$. Hence \[\{ \lambda : w_+ < \vert \lambda \vert < w_-\} \subseteq \sigma_p(F_w). \]

{\bf{Case 1.B}} Assume $w_+=0.$ We distinguish the cases $w_- \neq 0$ and $w_-=0$. If $w_- \neq 0$, and if $\lambda \in \sigma_p(F_w)$, from the previous computation and the fact that $\vert \lambda \vert \geq 0=w_+$, it follows that $\sigma_p(F_w) \subseteq \{ \lambda : 0 \leq \vert \lambda \vert \leq w_-\}$. Moreover, a similar computation shows that  $\{ \lambda : 0 < \vert \lambda \vert < w_-\} \subseteq \sigma_p(F_w)$. \\
If $w_-=0$, note that $\{ \lambda : w_+ < \vert \lambda \vert < w_-\} = \emptyset \subseteq \sigma_p(F_w).$ Moreover, if $\lambda \in \sigma_p(F_w)$, a computation similar to above shows that $\vert \lambda \vert =0$, i.e. $\sigma_p(F_w) \subseteq \{0\}.$ \\

For the proof of (2), just note that if $w_- < w_+$, then, arguing as above, we obtain $\sigma_p(F_w)=\emptyset$.\\

(3) Assume $w_- \leq w_+$. As in (1), $\lambda \in \sigma_p(B_w)$ if and only if there exists $\{x_n\}_{n \in \mathbb Z} \in \ell^{p}(\mathbb Z) \setminus \{0\}$ such that $B_w(\{x_n\}_{n \in \mathbb Z})=\{\lambda x_n\}_{n \in \mathbb Z}$. On the other hand, $B_w(\{x_n\}_{n \in \mathbb Z})=\{w_{n+1}x_{n+1}\}_{n \in \mathbb Z}$.
Hence, the component of $x=\{ x_n\}_{n \in \mathbb Z}$ are such that \[w_{n+1} x_{n+1}=\lambda x_{n} \text{ for each } n \in {\mathbb Z}.\] Therefore, for each $n \in \mathbb N$ we have: \[ x_n=\frac{\lambda ^n}{w_1 \cdots w_{n}} x_0\,\, \text{ and } \,\, x_{-n}=\frac{w_{-n+1}\cdots w_{0}}{\lambda ^n} x_0.\]
We want to determine $\lambda$ such that $\Vert x \Vert_p<\infty$. Arguing as in (1), with a similar computation we obtain that, if $w_- \leq w_+$ then  \[ \{ \lambda : w_- < \vert \lambda \vert < w_+\} \subseteq \sigma_p(B_w)\subseteq \{ \lambda : w_- \leq \vert \lambda \vert \leq w_+\}.\]
The implication (4) follows noting that if $w_- > w_+$, then, by arguing as in (3), we obtain $\sigma_p(B_w)=\emptyset$.
\end{proof}

\begin{prop} \label{corspect}
Let $\{w_n\}_{n \in \mathbb Z}$ be a bounded positive weight sequence. Let \[ \lim_{n \rightarrow +\infty} w_n=w_+ \text{ and }  \lim_{n \rightarrow \infty} w_{-n}=w_-.\]  
Let $T$ be the bilateral weighted shift $F_w$ or $B_w$, on $X=\ell^p({\mathbb Z}) $, $1\leq p < \infty$. Then the followings hold.

\begin{itemize}
\item[a)] If $T$ is invertible, then \[ \sigma(T)=\{ \lambda : \min\{w_-,w_+\} \leq \vert \lambda \vert \leq \max\{w_-,w_+\}\}.\] 
\item[b)] If $T$ is not invertible, then \[\sigma(T)=\{ \lambda : \vert \lambda \vert \leq \max\{w_-,w_+\} \}.\]
\end{itemize}
\end{prop}

\begin{proof}

By using Proposition \ref{propspectrumforward2}, as $\sigma(F_w)=\sigma(B_w)$, hence without loss of generality we can consider $T=F_w$.
We may assume that $\min\{w_-,w_+\} \neq 0,$ otherwise, by Proposition \ref{compactshift}, $F_w$ is compact and hence every nonzero $\lambda \in \sigma(F_w)$ is an eigenvalue of $F_w$, i.e., $\sigma(F_w)=\{0\}\cup \sigma_p(F_w)$ \cite[Proposition 7.1]{Conway2}.

Hence, let $\min\{w_-,w_+\} \neq 0.$ Note that, by hypotheses, for each $\epsilon >0$, there exist $\overline{n}, \overline{\overline{n}} \in {\mathbb N}$ such that, 
 \[w_+-\epsilon < {w}_n<  w_+ +\epsilon, \forall n\geq \overline{n}\]
and
 \[w_- -\epsilon < {w}_{-n} <  w_- +\epsilon, \forall n\geq \overline{\overline{n}}.\]
Let $N=\max\{ \overline{n}, \overline{\overline{n}}  \}$.
Note that 
\begin{align*}
\vert w_k \dots  w_{k+n-1} \vert^{\frac{1}{n}} &= ( w_k \dots  w_{k+n-1} )^{\frac{1}{n}} \\
&< \begin{cases}
\max\{ w_+ + \epsilon, w_- + \epsilon\}^{\frac{n}{n}}  & \text{if } \vert k \vert >N \\
(\underset{{k \in [-N,N]}}{\sup} w_k) ^{\frac{h +1}{n}} \cdot \max\{ w_+ + \epsilon, w_- + \epsilon\}^{\frac{n-h-1}{n}} & \text{if } \vert k \vert \leq N, k+h=N 
\end{cases} \\
&= 
\begin{cases}
\max\{ w_+ + \epsilon, w_- + \epsilon\}  & \text{if } \vert k \vert >N \\
(\underset{{k \in [-N,N]}}{\sup} w_k )^{\frac{h +1}{n}} \cdot \max\{ w_+ + \epsilon, w_- + \epsilon\}^{1-\frac{h+1}{n}} & \text{if } \vert k \vert \leq N, k+h=N 
\end{cases} 
\end{align*}
Hence,
\begin{align*}
r(F_w)&= \lim_{n \rightarrow \infty} \Vert F_w^n\Vert^{\frac{1}{n}}=  \lim_{n \rightarrow \infty} \sup_{k \in \mathbb Z} \vert w_k \dots  w_{k+n-1} \vert^{\frac{1}{n}} =\lim_{n \rightarrow \infty} \sup_{k \in \mathbb Z}( w_k \dots  w_{k+n-1})^{\frac{1}{n}} \\
&\leq \lim_{n \rightarrow \infty} \max \left \{ \max\{ w_+ + \epsilon, w_- + \epsilon\} ; (\underset{{k \in [-N,N]}}{\sup} w_k) ^{\frac{h +1}{n}} \cdot \max\{ w_+ + \epsilon, w_- + \epsilon\}^{1-\frac{h+1}{n}} \right \}\\
&= \max\{ w_+ + \epsilon, w_- + \epsilon\}\\
&= \max\{ w_+ , w_- \} + \epsilon \\ 
\end{align*}
Analogously,
\begin{align*}
\vert w_k \dots  w_{k+n-1} \vert^{\frac{1}{n}} &= ( w_k \dots  w_{k+n-1} )^{\frac{1}{n}} \\
&>
\begin{cases}
\min\{ w_+ - \epsilon, w_- - \epsilon\}^{\frac{n}{n}}  & \text{if } \vert k \vert >N \\
(\underset{{k \in [-N,N]}}{\inf}  w_k ) ^{\frac{h +1}{n}} \cdot \min\{ w_+ - \epsilon, w_- - \epsilon\}^{\frac{n-h-1}{n}} & \text{if } \vert k \vert \leq N, k+h=N 
\end{cases} \\
&=
\begin{cases}
\min\{ w_+ - \epsilon, w_- - \epsilon\}  & \text{if } \vert k \vert >N \\
(\underset{{k \in [-N,N]}}{\inf} w_k )^{\frac{h +1}{n}} \cdot \min\{ w_+ - \epsilon, w_- - \epsilon\}^{1-\frac{h+1}{n}} & \text{if } \vert k \vert \leq N, k+h=N 
\end{cases} 
\end{align*}
Hence,
\begin{align*}
r(F_w^{-1})&= \lim_{n \rightarrow \infty} \Vert F_w^{-n}\Vert^{\frac{1}{n}}=  \lim_{n \rightarrow \infty} \sup_{k \in \mathbb Z} \left \vert \dfrac{1}{w_k \dots  w_{k+n-1}} \right \vert^{\frac{1}{n}} =\lim_{n \rightarrow \infty} \sup_{k \in \mathbb Z} \left ( \dfrac{1}{w_k \dots  w_{k+n-1}} \right )^{\frac{1}{n}} \\
&\leq \lim_{n \rightarrow \infty} \max \left \{ \dfrac{1}{\min\{ w_+ - \epsilon, w_- - \epsilon\}} ;  \dfrac{1}{(\underset{{k \in [-N,N]}}{\inf}  w_k )^{\frac{h +1}{n}} \cdot \min\{ w_+ - \epsilon, w_- - \epsilon\}^{1-\frac{h+1}{n}}}\right  \}\\
&= \dfrac{1}{\min\{ w_+ - \epsilon, w_- - \epsilon\}} \\
&=\dfrac{1}{\min\{ w_+ , w_- \} - \epsilon} \\ 
\end{align*}
From the above computations it follows, by the arbitrariness of $\epsilon >0$, that $r(F_w) \leq \max\{ w_+ , w_- \}$ and $r(F_w^{-1})^{-1} \geq \min\{ w_+ , w_- \}$. \\
Hence, by using Proposition \ref{propspectrumforward2}, one can get that 
\begin{itemize}
\item if $F_w$ is invertible, then
 \[ \sigma(F_w)= \{ \lambda :  r(F_w^{-1})^{-1} \leq \vert \lambda \vert \leq r(F_w) \} \subseteq \{ \lambda : \min\{w_-,w_+\} \leq \vert \lambda \vert \leq \max\{w_-,w_+\}\},\] 
 \item if $F_w$ is not invertible, then 
 \[\sigma(F_w)= \{ \lambda :  \vert \lambda \vert \leq r(F_w) \} \subseteq \{ \lambda : \vert \lambda \vert \leq \max\{w_-,w_+\} \}.\] 
 \end{itemize}
As  $\sigma(F_w)$ ($\sigma(B_w)$) is compact and $\sigma_p(F_w) \subseteq \sigma(F_w)$ ($\sigma_p(B_w) \subseteq \sigma(B_w)$), then the thesis follows by applying statements (1) and (3) of Proposition \ref{propsigmap}. 
\end{proof}

\begin{rmk}
Let $X=\ell^p({\mathbb Z}) $, $1\leq p < \infty$, and let $T$ denote a bilateral weighted shift $F_w$ or $B_w$, on $X$, 
with $\{w_n\}_{n \in \mathbb Z}$ a positive weight sequence. Let \[\uplim_{n \rightarrow \infty}w_n=w_+^+; \uplim_{n \rightarrow \infty}w_{-n}=w_-^+\]
and \[\lowlim_{n \rightarrow \infty}w_n=w_+^-; \lowlim_{n \rightarrow \infty}w_{-n}=w_-^-.\] Note that the Proposition \ref{corspect} cannot be improved by replacing $\min \{w_-,w_+\}$ with $\min \{w_-^-,w_+^-\}$ and $\max \{w_-,w_+\}$ with $\max \{w_-^+,w_+^+\}$\\
To see this, take for example
\[w_{n} =\begin{cases} 3^{{(-1)}^{n}} & \text{if }  n \geq 0 \\
2^{{(-1)}^{n}} & \text{if } n \leq -1.
\end{cases}\] 
Then, by Remark \ref{rmknorm} and by the spectral radius formula,
 \begin{eqnarray*}
 r(T) &=&\lim_{n \rightarrow \infty}  
 {\left[\sup_{k \in \Bbb N}( w_kw_{k+1} \cdots w_{k+n-1}) \right]}^{\frac{1}{n}}=1 =r(T^{-1})^{-1}.
 \end{eqnarray*}
Clearly, $\{w_{-n}\}_{n \in \mathbb N}$  and $\{w_n\}_{n \in \mathbb N}$ are not regular, with $w_-^-=\frac{1}{2}$, $w_-^+=2$, $w_+^-=\frac{1}{3}$, $w_+^+=3$, and hence $\max \{w_-^+,w_+^+\}=3$ and $\min \{w_+^-,w_-^-\}=\frac{1}{3}.$
\end{rmk}

We conclude with a very simple example, as an application of Proposition \ref{corspect}. 

\begin{exmp} Let $\{w_{n}\}_{n \in {\mathbb Z}}$ be so defined: 
\[w_{n}= \left\{ \begin{array}{ll}
3 +\frac{1}{1+n} & \text{ if } n \geq 0\\
\frac{1}{2} -  \frac{1}{n}  & \text{ if } n \leq -1\\
\end{array}
\right.\] 
Let $T= F_{w}$ or $B_w$. Then, clearly, $T$ is invertible and $\max \{w_-,w_+\}=3$, $\min \{w_-,w_+\}=\frac{1}{2}$. Hence, by Proposition \ref{corspect}, 
\[ \sigma(T) =\left \{ \lambda : \frac{1}{2} \leq \vert \lambda \vert \leq 3 \right \}.\] 
\end{exmp}

\subsection*{Acknowledgements.} This research has been partially supported by the \emph{``Gruppo Nazionale per l'Analisi Matematica, la Probabilità e le loro Applicazioni dell'Is\-tituto Nazionale di Alta Matematica F. Severi''} and by the project \emph{Vain-Hopes} within the program \emph{Valere} of \emph{Università degli Studi della Campania ``Luigi Vanvitelli''}. It has also been partially accomplished within the \emph{UMI} Group \emph{TAA ``Approximation Theory and Applications''}

\bibliographystyle{siam}
\bibliography{biblio}

\Addresses

\end{document}